\pdfoutput=1 

\RequirePackage[l2tabu, orthodox]{nag}

\documentclass[12pt,reqno]{amsart}

\usepackage{hyperref}
\usepackage[shortlabels]{enumitem}
\usepackage{amsmath}
\usepackage{amsxtra}
\usepackage{amscd}
\usepackage{amsthm}
\usepackage{amsfonts}
\usepackage{amssymb}
\usepackage{eucal}
\usepackage[all]{xy}
\usepackage{graphicx}
\usepackage{tikz-cd}
\usepackage{mathrsfs}
\usepackage{subfiles}
\usepackage{mathpazo}
\usepackage[colorinlistoftodos, textsize=tiny]{todonotes}
\usepackage{morefloats}
\usepackage{pdfpages}
\usepackage{thm-restate}
\usepackage[utf8]{inputenc}
\usepackage{epigraph}
\usepackage{csquotes}
\usepackage[margin=1.5in]{geometry}
\usepackage{adjustbox}
\usepackage{colonequals}

\newcommand{\defi}[1]{\textsf{#1}} 

\usepackage{tikz-cd}
\newcommand{\Aff}{\mathbb{A}}
\newcommand{\C}{\mathbb{C}}

\newcommand{\G}{\mathbb{G}}

\newcommand{\PP}{\mathbb{P}}
\newcommand{\Q}{\mathbb{Q}}
\newcommand{\R}{\mathbb{R}}
\newcommand{\Z}{\mathbb{Z}}
\newcommand{\Qbar}{{\overline{\Q}}}

\newcommand{\kbar}{{\overline{k}}}

\newcommand{\calB}{\mathcal{B}}

\newcommand{\calH}{\mathcal{H}}

\newcommand{\MM}{\mathscr{M}}

\DeclareMathOperator{\Aut}{Aut}

\DeclareMathOperator{\Div}{Div}

\DeclareMathOperator{\NS}{NS}

\DeclareMathOperator{\Pic}{Pic}
\DeclareMathOperator{\PIC}{\textbf{Pic}}

\DeclareMathOperator{\specialK}{special}

\DeclareMathOperator{\id}{id}

\newcommand{\Directsum}{\bigoplus} 

\newcommand{\injects}{\hookrightarrow}
\newcommand{\intersect}{\cap} 

\newcommand{\pairing}{\langle \;,\; \rangle}
\renewcommand{\setminus}{-} 

\newcommand{\tensor}{\otimes} 

\newcommand{\Union}{\bigcup} 

\theoremstyle{plain}
\newtheorem{theorem}[subsection]{Theorem}
\newtheorem{proposition}[theorem]{Proposition}
\newtheorem{lemma}[theorem]{Lemma}
\newtheorem{corollary}[theorem]{Corollary}

\theoremstyle{definition}

\newtheorem{example}[theorem]{Example}
\newtheorem{question}[theorem]{Question}
\newtheorem{conjecture}[theorem]{Conjecture} 
\newtheorem{notation}[theorem]{Notation}
\newtheorem{ideaofproof}[theorem]{Idea of proof}

\theoremstyle{remark}
\newtheorem{remark}[theorem]{Remark}
\numberwithin{equation}{section}

\makeatletter
\g@addto@macro\bfseries{\boldmath} 
\makeatother

\def\listtodoname{List of Todos}
\def\listoftodos{\@starttoc{tdo}\listtodoname}

\usepackage{microtype}  

\makeatletter
\@namedef{subjclassname@2020}{%
  \textup{2020} Mathematics Subject Classification}
\makeatother

\begin{document}

\title[Obstructions to applying the Baker--Bilu method]{Obstructions to applying the Baker--Bilu method for determining integral points on curves}

\author{Aaron Landesman}
\address{Department of Mathematics, Massachusetts Institute of Technology, Cambridge, MA 02139-4307, USA}

\author{Bjorn Poonen}
\address{Department of Mathematics, Massachusetts Institute of Technology, Cambridge, MA 02139-4307, USA}
\email{poonen@math.mit.edu}
\urladdr{\url{http://math.mit.edu/~poonen/}}

\thanks{
B.P.\ was supported in part by National Science Foundation grant DMS-1601946 and Simons Foundation grants \#402472 and \#550033.}

\date{June 14, 2023}

\begin{abstract}
We prove that for every smooth projective integral curve $X$ of genus at least $2$ over $\C$, there exists $x \in X(\C)$ such that no connected finite \'etale cover of $X-\{x\}$ admits a nonconstant morphism to $\G_m$.
This has implications for the applicability of Baker's method to determining integral points on curves.
\end{abstract}

\maketitle

\section{Introduction}\label{S:introduction}

It is not known if there is an algorithm to find all the integer solutions to an arbitrary polynomial equation in two variables.
More generally, one can ask about solutions in a ring of $S$-integers of a number field.
Equivalently, one can ask about effectively bounding the height of $S$-integral points on an affine curve $U$.
One can reduce to the case where $U$ is smooth and geometrically integral; then $U = X \setminus R$ 
for some nice curve $X$ and some finite set $R$ of closed points of $X$. (Here, ``nice'' means smooth, projective, and geometrically integral.)
Let $g$ be the genus of $X$.
We assume that the Euler characteristic $\chi(U) = 2-2g-\#R$ is negative; in this case, the set of $S$-integral points is finite by Siegel's theorem~\cite{Siegel1929}, but the question is whether the result can be made effective.

Baker's method together with Dirichlet's $S$-unit theorem handles all cases with $g \le 1$ 
\cite{Baker1966I-IV,Baker1967I,Baker1967II,Baker1968,Baker-Coates1970}.
It also handles some $(X,R)$ with $g \ge 2$, such as those in which $X$ is a cyclic cover of $\PP^1$ and $R$ is an orbit of $\Aut(X/\PP^1)$ \cite{Baker1969}.
Bilu~\cite[Theorem~E]{Bilu1995} generalized Baker's approach to handle all cases in which $U$ after base field extension has a connected finite \'etale cover with a nondegenerate morphism to $\G_m \times \G_m$; here \defi{nondegenerate} means that the image is not contained in a coset of a proper algebraic subgroup.
Bilu's theorem begs the question:

\begin{question}
\label{question:Bilu}
Does \emph{every} smooth integral affine curve of genus at least $2$ over $\Qbar$ admit a connected finite \'etale cover with a nondegenerate morphism to $\G_m \times \G_m$?
\end{question}

We conjecture that the answer is \emph{no}.
In fact, we conjecture the following stronger statement:

\begin{conjecture}
\label{conjecture:not-one-morphism}
For \emph{any} nice curve $X$ of genus at least $2$, there exists $x \in X(\Qbar)$ such that the affine curve $X \setminus \{x\}$ has no connected finite \'etale cover with even a \emph{single} nonconstant morphism to $\G_m$.
\end{conjecture}

As evidence, we prove this statement with $\C$ in place of $\Qbar$:

\begin{theorem}
\label{theorem:no-finite-etale-cover-over-C}
Let $X$ be a nice curve of genus at least $2$ 
over $\C$.
Then there exists $x \in X(\C)$ such that $X\setminus \{x\}$ has no connected finite \'etale cover with a nonconstant morphism to $\G_m$.
\end{theorem}

Thus if, for some $g \ge 2$, \autoref{question:Bilu} has a positive answer for all curves of genus $g$ over $\Qbar$, it cannot be because of a universal algebraic construction of a cover of degree depending only on $g$, as was the case for $g=1$ --- instead the degree would have to depend on the height of the curve as well, as happens in Belyi's theorem \cite{Belyi1979}.

\begin{remark}
For a smooth integral affine curve $U = X \setminus  R$ as above,
having a nonconstant morphism $U \to \G_m$
is equivalent to having a nontrivial integer relation 
between the classes of the points of $R$ in $\Pic(X)$.
\end{remark}

For any nice curve $X$ of genus at least $2$ 
over an algebraically closed field $K$,
let $X(K)_{\specialK}$ be the set of $x \in X(K)$ such that $X\setminus \{x\}$ has a
connected finite \'etale cover with a nonconstant morphism to $\G_m$.
\autoref{theorem:no-finite-etale-cover-over-C} is an immediate consequence of
the following, which we prove in \autoref{S:proof of theorem}.

\begin{theorem} 
\label{theorem:special-set-is-countable}
Let $K$ be an algebraically closed field of characteristic~$0$.
Let $X$ be a nice curve of genus at least $2$ over $K$.
Then $X(K)_{\specialK}$ is countable.
\end{theorem}

One approach to proving a negative answer to \autoref{question:Bilu} (over $\Qbar$) might be to prove that for a nice curve $X$ of genus at least $2$ over $\Qbar$, the set $X(\Qbar)_{\specialK}$ is of bounded height.
Our \autoref{theorem:bounded-height} implies the slightly weaker result that the set of $x \in X(\Qbar)$ such that $X\setminus \{x\}$ admits such a cover of degree $\le d$ is a set of bounded height.
Another approach to \autoref{question:Bilu} might be to prove that $X(\Qbar)_{\specialK}$ is not $p$-adically dense in $X(\Qbar)$.
Perhaps $X(\Qbar)_{\specialK}$ is even \emph{finite} for most $X$.

\begin{example}[Ihara, Tamagawa]
If $X$ is of genus~$2$, then $X(\Qbar)_{\specialK}$ is infinite, because it contains the set called $M$ just before Proposition~6.1 in \cite{Bilu1995}.
(Bilu received this example in a letter from Yasutaka Ihara, who wrote that it emerged in a discussion with Akio Tamagawa.)
\end{example}


\begin{remark}
	\label{remark:}
	\autoref{theorem:no-finite-etale-cover-over-C} is related to a
	number of other conjectures in algebraic geometry and geometric topology.

	Prill's problem asked whether there is a nice curve $X$ of genus at
	least $2$ and a finite cover $f\colon Y \to X$ such that $H^0(Y, \mathscr
	O_Y(f^{-1}(x))) \geq 2$ for every $x \in X$ \protect{\cite[p. 268, Chapter VI, Exercise
	D]{ACGH:I}}.
	Call such a cover {\em Prill exceptional}.
	Any Prill exceptional cover has
	the property that for each $x \in X$, the curve $Y \setminus  f^{-1}(x)$ has a nonconstant map to $\Aff^1$.
	It is known that for every genus $2$ curve $X$ over $\mathbb C$,
 there is a finite \'etale Prill exceptional cover of $X$ \cite{landesmanL:prill}, but it remains wide open whether there exists any $X$ of genus at least $3$ with a Prill exceptional cover.

	Moreover, by \cite[Lemma~5.5]{landesmanL:applications-putman-wieland}, any Prill exceptional
	cover $f: Y \to X$ of a general curve $X$ of genus $g$ gives a
	counterexample to another conjecture in geometric topology, the
	Putman--Wieland conjecture,
	as stated in \cite[Conjecture 1.2]{putmanW:abelian-quotients}.
	In fact, \cite[Lemma~6.10]{landesmanL:introduction-putman-wieland} implies that the Putman--Wieland conjecture is equivalent to a statement
	about maps from covers of curves to abelian varieties (in place of 
	$\mathbb G_m$ as in \autoref{theorem:no-finite-etale-cover-over-C}): the
	statement is that for each $g \ge 2$, $n \ge 0$, and abelian variety $A$ over $\mathbb C$, a general $n$-pointed genus $g$ curve $X$ has no finite cover $Y \to X$ branched only over the $n$ points
	with a nonconstant map $Y \to A$.
	This statement is false for $g=2$ \cite[Theorem 1.3]{markovic} but open for $g \geq 3$.
	The Putman--Wieland conjecture, in turn, is closely related to
	another longstanding conjecture in geometric topology, Ivanov's
	conjecture \cite[Theorem~1.3]{putmanW:abelian-quotients}.
\end{remark}

\section{Proof of \autoref{theorem:special-set-is-countable}}
\label{S:proof of theorem}

\begin{ideaofproof}
Fix $K$ and $X$ as in \autoref{theorem:special-set-is-countable}.
First, we reduce to proving countability of the set $X(K)_{\specialK,G}$ of $x \in X(K)$ such that $X \setminus \{x\}$ has a $G$-\emph{Galois} finite \'etale cover with a nonconstant morphism to $\G_m$, for each finite group $G$.
Next, we construct a moduli space $M$ parametrizing $G$-covers $f \colon Y \to X$ branched at a varying point $x \in X$, together with a universal curve $S \to M$. We show that if there is a relation between the points in $f^{-1}(x)$ for a very general $x$, then there is a relation between certain divisors in $S$. Finally, we rule out such a relation by showing that the intersection matrix of these divisors is invertible.
\end{ideaofproof}


\begin{remark}
	\label{remark:suffices-countable}
	Since every finite \'etale cover is dominated by a Galois finite \'etale cover, $\Union_G X(K)_{\specialK,G} = X(K)_{\specialK}$.
From now on, we fix $G$.
It remains to prove that $X(K)_{\specialK,G}$ is countable.
\end{remark}

We next construct a moduli space of $G$-covers of $X$ branched over at most one point.

\begin{lemma}
	\label{lemma:moduli}
There is a finite-type $K$-scheme $\MM$ parametrizing
$(x,Y,f,\calB,y_1,\ldots,y_n)$, where $x$ is a point of $X$, $Y$ is a nice curve, $f \colon Y \to X$ is a Galois cover with Galois group $G$ that is \'etale above $X\setminus \{x\}$ (at least), and $\calB$ is a basis for $J[3]$, where $J$ is the Jacobian of $Y$,
and $y_1,\ldots,y_n$ are the distinct points of $f^{-1}(x)$
(here $n$ is constant on each irreducible component of $\MM$).
Each irreducible component of $\MM$ is a nice curve.
\end{lemma}

\begin{proof}
	This essentially follows from \cite[Theorem~4]{wewers:thesis}.	
Let the $D \subset X \to S$ of Wewers
be the diagonal $\Delta \subset X \times X \to X$.
	By \cite[Theorem~4]{wewers:thesis}, there is a finite-type algebraic stack
	$\calH = \calH_{X \times X}^G(G)$ over $X$ such that for any $S' \to X$,
	the groupoid $\calH(S')$ parameterizes $G$-Galois finite locally free covers $Y'
	\to S' \times_{X} (X \times X)$ that are tamely ramified above $S' \times_X
	\Delta$ and \'etale elsewhere.
	By \cite[Theorem~4]{wewers:thesis}, $\calH \to X$ is
	\'etale.  It follows from \cite[Theorem~3.2.4]{wewers:thesis} that $\calH \to X$ is
	proper.

	By adding the data of $\calB$ and the sections $y_1, \ldots, y_n$ 
	to our moduli stack, we obtain a finite
	\'etale cover $\MM$ of $\calH$.
	Each groupoid $\MM(S')$ is a setoid because $\Aut Y \to \Aut J[3]$ is injective \cite[10.5.6]{Katz-Sarnak1999}, and hence $\MM$ is represented by an algebraic space.
	Since $\MM \to \calH$ is
	finite \'etale and $\calH \to X$ is proper \'etale, $\MM$ is
	proper \'etale over $X$. By the previous two sentences, $\MM$ is finite \'etale over $X$ and is
	therefore a scheme. 
	Since $\MM$ is finite \'etale over the nice curve $X$, each 
	irreducible component of $\MM$ is a nice curve.
\end{proof}

\begin{notation}
\label{notation:relative-curve}
Let $M$ be an irreducible component of $\MM$.
Then $M$ is a nice curve.
Let $\eta$ be the generic point of $M$.
Let $\pi \colon S \to M$ be the universal morphism whose fiber above
$m = (x,Y,f,\calB,y_1,\ldots,y_n) \in M$ is $Y$.
Thus $S$ is a nice surface, and $\pi \colon S \to M$ is a relative curve.
Let $h \colon S \to X$ be the morphism whose restriction to each fiber $Y$ of $\pi$ is the map $f \colon Y \to X$.
Let $\mu\colon M \to X$ be $(x,Y,f,\calB,y_1,\ldots,y_n) \mapsto x$.
Let $e$ be the positive integer such that for any $m \in M$, the ramification index of the corresponding map $f \colon Y \to X$ at any point above $x$ is $e$.
Let $s_1,\ldots,s_n \colon M \to S$ be the sections such that $s_i(m)=y_i$ for each $i$.
Let $D_i=s_i(M) \in \Div(S)$.
Let $F \in \Div(S)$ be a closed fiber of $S \to M$.
Let $\NS(S)$ be the N\'eron--Severi group of $S$.
\end{notation}

\begin{proposition}
	\label{proposition:independence-of-sections}
 The classes of $D_1,\ldots,D_n,F$ in $\NS(S)$ are $\Z$-independent.
\end{proposition}
\begin{proof}
	It suffices to prove that the intersection matrix is nonsingular.
We know $D_i \cdot D_j = 0$ for $i \neq j$, $D_i \cdot F = 1$, $F \cdot
F = 0$, and $D_i^2 = d$ for some $d$ independent of $i$, because $G$ acts
transitively on $D_1,\ldots, D_n$.
Then the intersection matrix is
	\begin{align*}
		\begin{pmatrix}
			d & 0 & 0 & \cdots & 0 & 1 \\
			0 & d& 0 & \cdots & 0 &1 \\
			0 & 0& d & \cdots & 0 & 1 \\
			\vdots & \vdots & \vdots & \ddots & \vdots & \vdots \\
			0 & 0 & 0 & \cdots & d & 1 \\
			1 & 1&  1 & \cdots & 1 & 0
		\end{pmatrix},
	\end{align*}
	which has determinant $- n d^{n-1}$.
	In \autoref{lemma:self-intersection}, we verify that $d \neq 0$.
\end{proof}

    \begin{lemma}
		\label{lemma:self-intersection}
		We have $D_i^2  = (\deg \mu)(2-2g)/e < 0$.
	\end{lemma}
 
	\begin{proof}
	Observe that $\mu = h \circ s_i$ because
	$\mu(x,Y,f,\calB,y_1,\ldots,y_n) = x = h(y_i)$ and $y_i = s_i(x,Y,f,\calB,y_1,\ldots,y_n)$.
Let $\delta \colon X \to X \times X$ be the diagonal.
	The diagram
	\begin{equation}
		\label{equation:big}
		\begin{tikzcd} 
			S \ar{r}{(\pi, h)} & M \times X \ar {r}{\mu \times \id}
			& X \times X \\
			\qquad &M \ar{ul}{s_i} \ar[swap]{u}{(\id, \mu)} \ar{r}{\mu} \ar {r} &
			X \ar[swap]{u}{\delta}
	\end{tikzcd}\end{equation}
	 commutes, since $\pi \circ s_i = \id$ and $h \circ s_i = \mu$.

Let $\Delta=\delta(X) \subset X \times X$ be the diagonal.  
We will pull back the class of $\Delta$ to a line bundle on $M$ along the two outer paths in~\eqref{equation:big}.
First, we determine the pullback of $\Delta$ to $S$ by calculating its restriction to each fiber of $\pi \colon S \to M$.
If $m = (x,Y,f,\calB,y_1,\ldots,y_n) \in M$, 
so $Y$ is the fiber of $S \to M$ over $m$, 
then the composition 
\[
   Y \injects S \xrightarrow{(\pi,h)} M \times X \xrightarrow{\mu \times \id} X \times X
\]
is $Y \xrightarrow{(x,f)} X \times X$ (constant first coordinate),
so $(\pi,h)^* (\mu \times \id)^* \Delta |_Y = f^{-1}(x) = e \sum_{j=1}^n y_j$ in $\Div Y$.
Thus $(\pi,h)^* (\mu \times \id)^* \Delta = e \sum_{j=1}^n D_j$ in $\Div S$, and pulling this back to $M$ along $s_i$ yields $D_i \cdot (e \sum_{j=1}^n D_j) \in \Div M$,
which has degree $e D_i^2$.
On the other hand, the class of $\Delta \subset X \times X$ pulls back to a divisor of degree $\Delta \cdot \Delta = 2-2g$ on $X$, which pulls back to a divisor of degree $(\deg \mu)(2-2g)$ on $M$.
Thus
\[
  e D_i^2 = (\deg \mu)(2-2g),\qedhere
\]
which is negative, since $g \ge 2$.
\end{proof}

 \begin{corollary}
 \label{corollary:generic-independence}
 The images of $D_1,\ldots,D_n$ in $\Pic(S_\eta)$ are $\Z$-independent.
 \end{corollary}

 \begin{proof}
 If there were a relation, the exact sequence 
 \begin{equation}
		\label{equation:}
		\begin{tikzcd}
			 \displaystyle \Directsum_{t \in M(K)} \mathbb Z \cdot \pi^{-1}(t)
			\ar {r} & \Pic(S) \ar {r} & \Pic(S_\eta) \ar {r}
			& 0
	\end{tikzcd}\end{equation}
	would give a relation between $D_1,\ldots,D_n,F$ in $\NS(S)$,
 since the class of each fiber $\pi^{-1}(t)$ in $\NS(S)$ is $F$.
 \end{proof}

\begin{proof}[End of proof of \autoref{theorem:special-set-is-countable}]
As mentioned in \autoref{remark:suffices-countable}, it suffices to show
$X(K)_{\specialK,G}$ is countable.
It is the image under $\MM \to X$ of
\begin{equation}
\label{E:M(K)'}
   \MM(K)' \colonequals \{ (x,Y, \ldots) \in \MM(K) : \textup{$Y$ has a nonconstant morphism to $\G_m$} \},
\end{equation}
which is the finite union of sets $M(K)' \colonequals \MM(K)' \intersect M$.
Each $M(K)'$ is the union over nonzero $a = (a_1, \ldots, a_n) \in
\mathbb Z^n$ of
\begin{align*}
	V_a := \{ (x,Y,f,\calB,y_1,\ldots,y_n) \in M(K) : \sum_{i=1}^n a_i y_i =0 \textup{ in $\Pic(Y)$}\},
\end{align*}
so it suffices to prove that $V_a$ is finite.
The set $V_a$ is the zero locus of a section of the relative Picard scheme $\PIC_{S/M} \to M$, so $V_{a} \subset M(K)$ is closed.
If $V_a = M(K)$, then $\sum_{i=1}^n a_i D_i = 0$ in $\Pic(S_\eta)$, contradicting \autoref{corollary:generic-independence}.
Hence $V_a$ is finite.
\end{proof}

\section{Bounded height}
\label{section:bounded-height}

In this section, we set $K=\Qbar$ and prove in \autoref{theorem:bounded-height} that each set $X(\Qbar)_{\specialK,G}$ is of \emph{bounded height}.

First, we introduce notation for \autoref{theorem:canonical-heights}, which records known results about heights and specialization.
Let $\pi \colon S \to M$ be a morphism from a nice surface to a nice curve over $\Qbar$.
For each field extension $L \supset \Qbar$ and $t \in M(L)$, let $S_t = \pi^{-1}(t)$; if $S_t$ is a \emph{nice} curve, let $J_t$ be its Jacobian, a principally polarized abelian variety over $L$.
By a \defi{fibral component} we mean an irreducible component of $S_t$ for some $t \in M(\Qbar)$.
Let $\eta$ be the generic point of $M$.  We assume that $S_\eta$ is a \emph{nice} curve (over the function field $k(M)$).

Let $k$ be a field with a product formula as in \cite[\S
2.1]{serre:lectures-on-the-mordell-weil-theorem}. 
Given a polarized abelian variety $A$ over $k$, 
one can enlarge $k$ so that the polarization arises from a symmetric divisor on $A$
and then define a canonical height pairing on $A(\kbar)$, or its subgroup $A(k)$, as in \cite[pp.~200--201]{Silverman1983}. 
Applying this to $J_\eta$ over $k(M)$ yields a (geometric) canonical height pairing $\pairing$ on $J_\eta(k(M)) = \Pic^0(S_\eta)$ or on $\Div^0(S_\eta)$.
For any $t \in M(\Qbar)$ such that $S_t$ is a nice curve, applying this to $J_t$ (over a number field to which it descends) yields a canonical height pairing $\pairing_t$ on $J_t(\Qbar) = \Pic^0(S_t)$ or on $\Div^0(S_t)$.

For each closed point $P \in S_\eta$, let $\overline{P}$ be its Zariski closure in $S$. 
Extend $\Z$-linearly to define $\overline{P} \in \Div S$ for any $P \in \Div(S_\eta)$;
then for any $t \in M(\Qbar)$, define the specialization $P_t \colonequals \overline{P}|_{S_t} \in \Div(S_t)$.
Let $h \colon M(\Qbar) \to \R$ be a Weil height associated to a divisor of nonzero degree on $M$, as in \cite[p.~205]{Silverman1983}.

\begin{theorem}
\label{theorem:canonical-heights}
Let $S$, $M$, $\pi$, $S_\eta$, $S_t$, $\pairing$, $\pairing_t$, and $h$ be as above.
\begin{enumerate}[\upshape (a)]
\item \label{I:fibral adjustment}
   Let $P \in \Div^0(S_\eta)$.
Then there exists a $\Q$-linear combination $\Phi$ of fibral components such that $D_P \colonequals \overline{P} + \Phi \in \Div(S) \tensor \Q$ is orthogonal to each fibral component.
Moreover, $\Phi$ is unique modulo $\pi^* \Div(M) \tensor \Q$.
\item \label{I:canonical height and intersection pairing}
   For $P,P' \in \Div^0(S_\eta)$, we have $\langle P,P' \rangle = - D_P \cdot D_{P'}$.
\item \label{I:height ratio}
   Fix $P,Q \in \Div^0(S_\eta)$. Then $\langle P_t,Q_t \rangle_t / h(t) \to \langle P,Q \rangle$ as $h(t) \to \infty$.
\item \label{I:independence of specializations}
   Let $P_1,\ldots,P_n \in \Div^0(S_\eta)$.
If the matrix $(\langle P_i,P_j \rangle)$ is positive definite, 
then the set of $t \in M(\Qbar)$ such that $P_{1,t},\ldots,P_{n,t}$ are $\Z$-dependent in $\Pic^0(S_t)$ is of bounded height.
\end{enumerate}
\end{theorem}

\begin{proof}\hfill
\begin{enumerate}[\upshape (a)]
\item Let $t \in M(\Qbar)$.
Let $F_1,\ldots,F_n$ be the irreducible components of $S_t$.
Let $F$ be the class of $S_t$ in $\Div(S)$.
The existence and uniqueness of the part of $\Phi$ supported above $t$ follows from the nondegeneracy of the intersection pairing on $(\Directsum \Q F_i) / \Q F$, as in \cite[end of \S3]{Gross1986} (the finiteness of the residue field assumed there is not needed for this).
\item This too follows from the same argument as in \cite[end of \S3]{Gross1986}.  (The relation between canonical heights and intersection numbers was discovered earlier in the context of an abelian variety and its dual, by Manin~\cite{Manin1964} and N\'eron~\cite{Neron1965}.)
\item This is a version of Silverman's specialization theorem, which extended ideas of Dem$'$janenko \cite{Demjanenko1968} and Manin \cite{Manin1969}.  Specifically, the case $P=Q$ is contained in \cite[Theorem~B]{Silverman1983}.  The general case follows from this case applied to $P$, $Q$, and $P+Q$.
\item
If $h(t)$ is sufficiently large, then \eqref{I:height ratio} implies that $(\langle P_{i,t},P_{j,t} \rangle)$ is positive definite too, so $P_{1,t},\ldots,P_{n,t}$ are $\Z$-dependent in $\Pic^0(S_t)$.\qedhere
\end{enumerate}
\end{proof}

\begin{theorem}
\label{theorem:bounded-height}
Let $X$ be a nice curve of genus at least $2$ over $\Qbar$.
Then for each finite group $G$, the set $X(\Qbar)_{\specialK,G}$ is of bounded height.
\end{theorem}

\begin{proof}
Let $S \to M$ and $D_1,\ldots,D_n$ be as in \autoref{notation:relative-curve}.
For $i=2,\ldots,n$, let $P_i = (D_i-D_1)|_{S_\eta}$.
Since $S \to M$ has irreducible fibers, we can always take $\Phi=0$ in \autoref{theorem:canonical-heights}\eqref{I:fibral adjustment}.
By \autoref{lemma:self-intersection}, the matrix $(D_i\cdot D_j)_{1 \leq i,j
\leq n}$ is $dI$ for some $d<0$,
so $((D_i-D_1) \cdot( D_j-D_1))_{2 \le i,j \le n}$ is negative definite too.
By \autoref{theorem:canonical-heights}\eqref{I:canonical height and intersection pairing}, 
$(\langle P_i,P_j \rangle)$ is \emph{positive} definite.
For $t \in M(\Qbar)$, the following are equivalent:
\begin{itemize}
\item $t$ belongs to the subset $M(\Qbar)'$ defined in~\eqref{E:M(K)'};
\item $D_1|_{S_t},\ldots,D_n|_{S_t}$ are $\Z$-dependent in $\Pic(S_t)$;
\item $P_2,\ldots,P_n$ are $\Z$-dependent in $\Pic^0(S_t)$.
\end{itemize}
By \autoref{theorem:canonical-heights}\eqref{I:independence of specializations},
the set of such $t$ is of bounded height.
Thus $M(\Qbar)'$ is of bounded height,
and so is its image in $X(\Qbar)$.
Taking the union over the finitely many irreducible components $M$ of $\MM$ shows that $X(\Qbar)_{\specialK,G}$ is of bounded height.
\end{proof}

\begin{remark}
To prove the weaker statement that $X(\Qbar) \setminus X(\Qbar)_{\specialK,G}$
is infinite, one could use N\'eron's specialization theorem \cite[Th\'eor\`eme~6]{Neron1952} in place of \cite[Theorem~B]{Silverman1983}.
\end{remark}

\begin{question}
	\label{question:}
	Can we refine the proof of \autoref{theorem:bounded-height} to obtain a height bound that is uniform in $G$?
If so, then \autoref{question:Bilu} has a negative answer.
\end{question}


\section*{Acknowledgments} 

Part of the research for this article was carried out at the 2023 ``Arithmetic, Algebra, and Algorithms'' workshop at the ICMS in Edinburgh.
We heartily thank Umberto Zannier for explaining to us the relevance of Silverman's article \cite{Silverman1983} (building on the specialization theorems of Dem$'$janenko and Manin) for proving \autoref{theorem:bounded-height}, and for spending the time to write up for us a new, more elementary proof of the specialization theorem in the case of constant abelian schemes, inspired by the ``alternative proof'' of \cite[Theorem~1]{Bombieri-Masser-Zannier2008}.
We also thank Yuri Bilu, Noam Elkies, Benedict Gross, Joseph Harris, Will Sawin, Joseph Silverman, and Yuri Zarhin for helpful correspondence.

\bibliographystyle{alpha}
\bibliography{master}

\end{document}